\documentclass{article}

\usepackage{
	amsmath, 
	amsthm, 
	amssymb,
	fancyhdr, 
	setspace,
	tikz-cd,
}
\usepackage{upgreek}

\usepackage{lscape}
\usepackage{color}
\usepackage[color,matrix,arrow]{xy}
\usepackage{setspace}
\xyoption{all}

\DeclareMathOperator{\Ext}{Ext}

\DeclareMathOperator{\pd}{pd}

\DeclareMathOperator{\add}{add}

\DeclareMathOperator{\End}{End}

\usepackage{avant}
\usepackage{amsmath, amssymb, amsfonts, color, tikz, bbm, txfonts, euscript}
\usepackage{amssymb}
\usepackage{enumerate}
\usepackage{bm}
\newtheorem{theorem}{Theorem}[section]
\newtheorem{lemma}[theorem]{Lemma}

\theoremstyle{definition}
\newtheorem{mydef}[theorem]{Definition}

\begin{document}

\thispagestyle{empty}

\title{Short Proof Of A Conjecture Concerning Split-By-Nilpotent Extensions}
\author{Stephen Zito\thanks{The author was supported by the University of Connecticut-Waterbury}}
        
\maketitle

\begin{abstract}
Let $C$ be a finite dimensional algebra with $B$ a split extension by a nilpotent bimodule $E$.  We provide a short proof to a conjecture by Assem and Zacharia concerning properties of $\mathop{\text{mod}}B$ inherited by $\mathop{\text{mod}}C$.  We show if $B$ is a tilted algebra, then $C$ is a tilted algebra.
\end{abstract}

\section{Introduction}
Let $C$ and $B$ be finite dimensional algebras over an algebraically closed field $k$ such that there exists a split surjective morphism $B\rightarrow C$ whose kernel $E$ is contained in the radical of $B$: we then say $B$ is a split-by-nilpotent extension of $C$ by $E$.  One can study which properties of $\mathop{\text{mod}}B$ are inherited by $\mathop{\text{mod}}C$.  This question was investigated by Assem and Zacharia in $\cite{AZ}$.  They considered several classes of algebras that have been extensively studied in representation theory including quasi-tilted algebras, shod algebras, weakly shod algebras, left and right glued algebras, and laura algebras.  They proved if $B$ is one of the aforementioned algebras then so is $C$.  They also conjectured if $B$ is tilted, then so is $C$.  In this note, we provide an affirmative answer.
\begin{theorem}
Let $B$ be a split extension of $C$ by a nilpotent bimodule $E$.  If $B$ is a tilted algebra, then $C$ is a tilted algebra.
\end{theorem}
Throughout this paper, we use freely and without further reference properties of the module categories and Auslander-Reiten sequences as can be found in $\cite{ASS}$.  All algebras are assumed to be finite dimensional over an algebraically closed field $k$.  For an algebra $C$, we denote by $\tau_C$ the Auslander-Reiten translation in $\mathop{\text{mod}}C$.  All modules are considered finitely generated right modules.  Given $C$, we denote the Auslander-Reiten quiver by $\Gamma_C$.

 \subsection{Split extensions and extensions of scalars}
We begin this section with the formal definition of a split extension.
\begin{mydef}
 Let $B$ and $C$ be two algebras.  We say $B$ is a $\emph{split extension}$ of $C$ by a nilpotent bimodule $E$ if there exists a short exact sequence of $B$-modules
 \[
 0\rightarrow E\rightarrow B\mathop{\rightleftarrows}^{\mathrm{\pi}}_{\mathrm{\sigma}} C\rightarrow 0
\]
where $\pi$ and $\sigma$ are algebra morphisms, such that $\pi\circ\sigma=1_C$, and $E=\ker\pi$ is nilpotent.  
\end{mydef}
 A useful way to study the module categories of $C$ and $B$ is a general construction via the tensor product, also know as $\emph{extension of scalars}$, that sends a $C$-module to a particular $B$-module.
 \begin{mydef}
 Let $C$ be a subalgebra of $B$ such that $1_C=1_B$, then
 \[
 -\otimes_CB:\mathop{\text{mod}}C\rightarrow\mathop{\text{mod}}B
 \]
 is called the $\emph{induction functor}$, and dually
 \[
 D(B\otimes_CD-):\mathop{\text{mod}}C\rightarrow\mathop{\text{mod}}B
 \]
 is called the $\emph{coinduction functor}$.  Moreover, given $M\in\mathop{\text{mod}}C$, the corresponding induced module is defined to be $M\otimes_CB$, and the coinduced module is defined to be $D(B\otimes_CDM)$.  
 \end{mydef}
  The following result shows induction of an indecomposable $C$-module is an indecomposable $B$-module. 
 \begin{lemma}$\emph{\cite[Lemma~1.2]{AM}}$
 \label{indsummand}
 Let $M$ be a $C$-module.  There exists a bijective correspondence between the isomorphism classes of indecomposable summands of $M$ in $\mathop{\text{mod}}A$, and the isomorphism classes of indecomposable summands of $M\otimes_CB$ in $\mathop{\text{mod}}B$, given by $M\rightarrow M\otimes_CB$.
 \end{lemma}
 In particular, applying lemma $\ref{indsummand}$ to $C_C$ and $C_C\otimes_CB_B\cong B_B$, there exists a bijective correspondence between the isomorphism classes of indecomposable projective $C$-modules and the isomorphism classes of indecomposable projective $B$-modules, given by $P\rightarrow P\otimes_CB$.
 
 Next, we state a description of the Auslander-Reiten translation of an induced module.
 \begin{lemma}$\emph{\cite[Lemma~2.1]{AM}}$
 \label{AM1}
 For a $C$-module M, we have
 \[
 \tau_B(M\otimes_CB)\cong\emph{Hom}_C(_BB_C,\tau_CM)
 \]
 \end{lemma}
This lemma is important because it allows us to use an adjunction isomorphism.  
 \begin{lemma}
\label{Adjunct}
 Let $M$ be a $C$-module, $M\otimes_CB$ the induced module, and let $X$ be any $B$-module.  Then we have 
\[
\emph{Hom}_B(X,\tau_B(M\otimes_CB))\cong\emph{Hom}_B(X,\emph{Hom}_C(_BB_C,\tau_CM)\cong\emph{Hom}_C(X\otimes_BB_C,\tau_CM)
\]
and
\[
\emph{Hom}_B(M\otimes_CB,X)\cong\emph{Hom}_C(M,\emph{Hom}_B(_CB_B,X)).
\]
\end{lemma}
\begin{proof}
These isomorphisms follow from Lemma $\ref{AM1}$ and the adjunction isomorphism.
\end{proof}
We note that $\_\otimes_BB_C$ and $\text{Hom}_B(_CB_B,\_)$ are two expressions for the forgetful functor $\mathop{\text{mod}}B\rightarrow\mathop{\text{mod}}C$.  If $X$ is a $B$-module, we will denote the $C$-module structure by $(X)_C$.

 \subsection{Tilted Algebras}
  We begin with the definition of a partial tilting module.
   \begin{mydef} Let $A$ be an algebra.  An $A$-module $T$ is a $\emph{partial tilting module}$ if the following two conditions are satisfied: 
   \begin{enumerate}
   \item[($\text{1}$)] $\pd_AT\leq1$.
   \item[($\text{2}$)] $\Ext_A^1(T,T)=0$.
   \end{enumerate}
   A partial tilting module $T$ is called a $\emph{tilting module}$ if it also satisfies the following additional condition:
   \begin{enumerate}
   \item[($\text{3}$)] There exists a short exact sequence $0\rightarrow A\rightarrow T'\rightarrow T''\rightarrow 0$ in $\mathop{\text{mod}}A$ with $T'$ and $T''$ $\in \add T$.
   \end{enumerate}
   \end{mydef}
 We now state the definition of a tilted algebra.
 \begin{mydef} Let $A$ be a hereditary algebra with $T$ a tilting $A$-module.  Then the algebra $C=\End_AT$ is called a $\emph{tilted algebra}$.
 \end{mydef}
There exists a characterization of tilted algebras which was obtained independently by Liu $\cite{LIU}$ and Skowro$\acute{\text{n}}$ski $\cite{SK}$.  It uses the concept of a section.
\begin{mydef}
Let $C$ be an algebra with $\Gamma_C$ its Auslander-Reiten quiver.  A connected full subquiver $\Sigma$ of $\Gamma_C$ is a $\emph{section}$ if the following conditions are satisfied:
\begin{enumerate}
\item[($\text{1}$)] $\Sigma$ contains no oriented cycles.
\item[($\text{2}$)] $\Sigma$ intersects each $\tau_C$-orbit exactly once.
\item[($\text{3}$)] If $x_0\rightarrow x_1\rightarrow \cdots \rightarrow x_t$ is a path in $\Gamma_C$ with $x_0,x_t\in\Sigma$, then $x_i\in\Sigma$ for all $i$ such that $0\leq i \leq t$.
\end{enumerate}
\end{mydef}
We recall that a $C$-module $M$ is $\it{faithful}$ if its right annihilator
   \[
   \mathop{\text{Ann}}M=\{c\in C~|~Mc=0\}.
   \]
   vanishes.
The well-known criteria of Liu and Skowro$\acute{\text{n}}$ski states that an algebra $C$ is tilted if and only if $\Gamma_C$ contains a component $\mathcal{C}$ with a faithful section $\Sigma$ such that $\text{Hom}_C(X,\tau_CY)=0$ for all modules $X$, $Y$ from $\Sigma$.

A recent characterization of tilted algebras was established in $\cite{JMS}$.  We recall that a $C$-module $M$ is $\text{sincere}$ if $\text{Hom}_C(C,M)\neq 0$ for every indecomposable summand of $C$.
\begin{theorem}$\emph{\cite[Theorem~1]{JMS}}$
\label{Sincere}
An algebra $C$ is tilted if and only if there exists a sincere module $M$ in $\mathop{\emph{mod}}C$ such that, for any indecomposable module $X$ in $\mathop{\emph{mod}}C$,  $\emph{Hom}_C(X,M)=0$ or $\emph{Hom}_C(M,\tau_CX)=0$.
\end{theorem}
In particular, let $T_{\Sigma}$ be the direct sum of modules lying on the section $\Sigma$.  It was shown in the proof of theorem $\ref{Sincere}$ that $T_{\Sigma}$ is a sincere module satisfying the conditions of the theorem.

\section{Proof of Theorem 1.1}
\begin{proof}   
Let $B$ be a split extension of $C$ by the nilpotent bimodule $E$.  Suppose $B$ is a tilted algebra.  By the criteria of Liu and Skowro$\acute{\text{n}}$ski, there exists a faithful section $\Sigma$ in $\Gamma_B$.  Let $T_{\Sigma}$ be the direct sum of modules lying on the section.  Since $T_{\Sigma}$ is faithful, it is sincere.  By theorem $\ref{Sincere}$, $\text{Hom}_B(X,T_{\Sigma})=0$ or $\text{Hom}_B(T_{\Sigma},\tau_BX)=0$ for every indecomposable module $X$ in $\mathop{\text{mod}}B$.  Consider $(T_{\Sigma})_C$ in $\mathop{\text{mod}}C$.  We know every projective module in $\mathop{\text{mod}}B$ is induced from a projective module in $\mathop{\text{mod}}C$ by lemma $\ref{indsummand}$.  Using lemma $\ref{Adjunct}$ and the fact that $T_{\Sigma}$ is sincere in $\mathop{\text{mod}}B$, $\text{Hom}_B(P\otimes_CB,T_{\Sigma})\cong\text{Hom}_C(P,(T_{\Sigma})_C)\neq 0$ for every indecomposable projective $C$-module $P$.  Thus, $(T_{\Sigma})_C$ is sincere in $\mathop{\text{mod}}C$.  
\par
Suppose $C$ is not tilted.  By theorem $\ref{Sincere}$ and the fact that $(T_{\Sigma})_C$ is sincere, we must have $\text{Hom}_C(Y,(T_{\Sigma})_C)\neq 0$ and $\text{Hom}_C((T_{\Sigma})_C,\tau_CY)\neq 0$ for some indecomposable $C$-module $Y$.  Consider the induced module $Y\otimes_CB$.  Lemma $\ref{indsummand}$ implies $Y\otimes_CB$ is indecomposable in $\mathop{\text{mod}}B$.  This further implies $\tau_B(Y\otimes_CB)$ is indecomposable.  Using lemma $\ref{Adjunct}$,
\[  
\text{Hom}_B(Y\otimes_CB,T_{\Sigma})\cong \text{Hom}_C(Y,(T_{\Sigma})_C)\neq 0
\]
and
\[
\text{Hom}_B(T_{\Sigma},\tau_B(Y\otimes_CB))\cong\text{Hom}_C((T_{\Sigma})_C,\tau_CY)\neq 0.
\]
This contradicts $\text{Hom}_B(X,T_{\Sigma})=0$ or $\text{Hom}_B(T_{\Sigma},\tau_BX)=0$ for every indecomposable module $X$ in $\mathop{\text{mod}}B$ and we conclude $C$ is a tilted algebra.

\end{proof}

\noindent Mathematics Faculty, University of Connecticut-Waterbury, Waterbury, CT 06702, USA
\it{E-mail address}: \bf{stephen.zito@uconn.edu}

\end{document}